\documentclass{article}%
\usepackage{graphicx}
\usepackage{amsmath}
\usepackage{amsfonts}
\usepackage{amssymb}%
\providecommand{\U}[1]{\protect\rule{.1in}{.1in}}
\providecommand{\U}[1]{\protect\rule{.1in}{.1in}}
\providecommand{\U}[1]{\protect\rule{.1in}{.1in}}
\providecommand{\U}[1]{\protect\rule{.1in}{.1in}}
\providecommand{\U}[1]{\protect\rule{.1in}{.1in}}
\providecommand{\U}[1]{\protect\rule{.1in}{.1in}}
\providecommand{\U}[1]{\protect\rule{.1in}{.1in}}
\providecommand{\U}[1]{\protect\rule{.1in}{.1in}}
\providecommand{\U}[1]{\protect\rule{.1in}{.1in}}
\providecommand{\U}[1]{\protect\rule{.1in}{.1in}}
\providecommand{\U}[1]{\protect\rule{.1in}{.1in}}
\providecommand{\U}[1]{\protect\rule{.1in}{.1in}}
\providecommand{\U}[1]{\protect\rule{.1in}{.1in}}
\providecommand{\U}[1]{\protect\rule{.1in}{.1in}}
\providecommand{\U}[1]{\protect\rule{.1in}{.1in}}
\providecommand{\U}[1]{\protect\rule{.1in}{.1in}}
\providecommand{\U}[1]{\protect\rule{.1in}{.1in}}
\providecommand{\U}[1]{\protect\rule{.1in}{.1in}}
\providecommand{\U}[1]{\protect\rule{.1in}{.1in}}
\providecommand{\U}[1]{\protect\rule{.1in}{.1in}}
\providecommand{\U}[1]{\protect\rule{.1in}{.1in}}
\providecommand{\U}[1]{\protect\rule{.1in}{.1in}}
\providecommand{\U}[1]{\protect\rule{.1in}{.1in}}
\providecommand{\U}[1]{\protect\rule{.1in}{.1in}}
\providecommand{\U}[1]{\protect\rule{.1in}{.1in}}
\providecommand{\U}[1]{\protect\rule{.1in}{.1in}}
\providecommand{\U}[1]{\protect\rule{.1in}{.1in}}
\providecommand{\U}[1]{\protect\rule{.1in}{.1in}}
\providecommand{\U}[1]{\protect\rule{.1in}{.1in}}

\newtheorem{theorem}{Theorem}
{}

\newtheorem{case}{Case}

\newtheorem{corollary}{Corollary}

\newtheorem{example}{Example}

\newtheorem{lemma}{Lemma}
{}

\newtheorem{summary}{Summary}
\newenvironment{proof}[1][Proof]{\textbf{#1.} }{\ \rule{0.5em}{0.5em}}

\begin{document}

\title{On the Spectrality of the Differential Operators with Periodic Coefficients}
\author{O. A. Veliev\\{\small \ Dogus University, }\\{\small Esenkent 34755, \ Istanbul, Turkey.}\\\ {\small e-mail: oveliev@dogus.edu.tr}}
\date{}
\maketitle

\begin{abstract}
In this paper, we establish a condition on the coefficients of differential
operators generated in $L_{2}(-\infty,\infty)$ by an ordinary differential
expression with periodic, complex-valued coefficients, under which the
operator is a spectral operator.



\end{abstract}

\section{\qquad Introduction and preliminary Facts}

Let $L$ be the differential operator generated in the space $L_{2}%
(-\infty,\infty)$ by the differential expression
\begin{equation}
(-i)^{n}y^{(n)}(x)+%
{\textstyle\sum\limits_{v=1}^{n}}
p_{v}(x)y^{(n-v)}(x), \tag{1}%
\end{equation}
where $n$ is an integer greater than $1$ and $p_{v}$, for $v=1,2,...n$, are
$1$-periodic complex-valued functions satisfying $\left(  p_{v}\right)
^{(n-v)}\in L_{2}\left[  0,1\right]  $. It is well-known that (see [3, 4, 6])
the spectrum $\sigma(L)$ of the operator $L$ is the union of the spectra
$\sigma(L_{t})$ of the operators $L_{t}$, for $t\in(-1,1]$, generated in
$L_{2}\left[  0,1\right]  $ by (1) and the boundary conditions
\begin{equation}
y^{(\mathbb{\nu})}\left(  1\right)  =e^{i\pi t}y^{(\mathbb{\nu})}\left(
0\right)  \tag{2}%
\end{equation}
for $\mathbb{\nu}=0,1,...,(n-1).$ The spectra $\sigma(L_{t})$ of the operators
$L_{t}$ consist of the eigenvalues called the Bloch eigenvalues of $L$.

This paper can be considered a continuation of [9], in which we established a
condition on the coefficients of differential operators generated by a
vectorial differential expression with periodic matrix coefficients, under
which the operator in question is asymptotically spectral. In particular, for
the scalar case, we proved that $L$ is an asymptotically spectral operator in
the following cases:

\begin{case}
$n$ is an odd number.
\end{case}

\begin{case}
$n$ is an even number and $\operatorname{Re}\int_{0}^{1}p_{1}(x)dx\neq0.$
\end{case}

In this paper, we obtain the following results.

\textbf{Result 1}. If $n$ is an odd number and
\begin{equation}
C\leq\ \pi^{2}2^{-n+1/2}, \tag{3}%
\end{equation}
then $L$ is a spectral operator, where
\[
C=%
{\textstyle\sum\limits_{v=2}^{n}}
{\textstyle\sum\limits_{s=0}^{n-v}}
\frac{\left(  n-v\right)  !\left\Vert \left(  p_{v}\right)  ^{(s)}\right\Vert
}{s!(n-v-s)!\pi^{v+s-2}}%
\]
and $\left\Vert \cdot\right\Vert $ is the $L_{2}\left[  0,1\right]  $-norm.

Note that, by the well-known substitution, expression (1) can be reduced to a
form in which $p_{1}$ is identically the zero function. Moreover, if $n$ is an
odd number, this substitution does not chance the behavior of $L.$ Therefore,
without loss of generality and to apply the results of [13] directly, we
assume in this case that \ $p_{1}$ is the zero function.

\textbf{Result 2. }If $n$ is an even number greater than $2$ and $p_{1}(x)=c$
for all $x$, where
\begin{equation}
c^{2}\geq\left(  \frac{1}{6}+\frac{2^{2n-4}}{\pi^{2}}\right)  C^{2}, \tag{4}%
\end{equation}
then $L$ is a spectral operator, where $c$ is a real nonzero constant and $C$
is defined in (3). In the case $n=2,$ the operator $T(c,q),$ generated in
$L_{2}(-\infty,\infty)$ by the expression $-y^{\prime\prime}+cy^{\prime}+qy,$
where $c$ is a nonzero real number and $q$ is a complex -valued, locally
square integrable, periodic function is a spectral operator if
\begin{equation}
\left\vert c\right\vert >\frac{1}{2}\left\Vert q\right\Vert . \tag{5}%
\end{equation}

Note that in [4], for $n>2,$ only the asymptotic spectrality of the operator
$L$ was investigated, whereas in this paper we consider the spectrality of
$L.$ The asymptotic spectrality considered in [4] and [9], using different
method, requires proving that the projections of $L$ corresponding to parts of
the spectrum lying in neighborhoods of infinity are uniformly bounded, whereas
spectrality requires that all spectral projections be uniformly bounded.
Therefore, in the case asymptotic spectrality, it is sufficient to analyze the
asymptotic formulas for the eigenvalues and eigenfunctions of $L_{t}$ for
$t\in(-1,1]$. However, in this paper, through a detailed investigation of all
Bloch eigenvalues, we obtain the above results concerning spectrality.
Moreover, in [4], for $n>2,$ the asymptotic spectrality of the operator $L$
was investigated by by imposing certain conditions on the distances between
the eigenvalues of $L_{t}$, while we prove the spectrality of $L$ by imposing
conditions solely on the $L_{2}\left[  0,1\right]  $-norm of the coefficients.

In the case $n=2,$ in [4], the spectrality of $T(c,q)$ is investigated by
imposing a condition on the supremum norm of $q$, which is applies only to
bounded function. In this paper, the spectrality of $L(c,q)$ is established by
imposing condition (5) solely on the $L_{2}\left[  0,1\right]  $-norm of $q,$
which is applicable to any locally square integrable, periodic function $q$.

\textbf{Result 1} is obtained in Section 2 by using the asymptotic spectrality
of $L$ proved in [9] along with some results from [13] on the localization of
all Bloch eigenvalues, which address only the case of odd order.

In Section 3, we consider the case where $n$ is an even number. The asymptotic
spectrality of $L$ established in [9] is also used in this case. However, the
investigations and methods for the localization of all Bloch eigenvalues
presented in [13] cannot be applied to the even-order case. Therefore, in
Section 3, we independently investigate all Bloch eigenvalues for even $n.$

It is important to note that when $n$ is even, the operator $L$ generated by
(1) is, in general, not a spectral operator. Furthermore, the smallness and
smoothness of the coefficients in (1) do not imply the spectrality of $L$, and
the condition on $p_{1}$ used in Case 2 is, in a certain sense, necessary. Let
us explain this for $n=2,$ that is, for the Schrodinger operator
$T(q):=T(0,q)$ generated by the expression $-y^{\prime\prime}+qy$ with complex
-valued potential $q.$ In [8], and [12, Sect. 3], we proved that if there
exists an associated function corresponding to some double Bloch eigenvalue,
then the projection about these eigenvalue are not uniformly bounded. Since
the existence of the associated functions is the widespread case for the
non-self-adjoint operator, in general, $T(q)$ is not a spectral operator.
Gesztezy and Tkachenko [2] proved two versions of a criterion for the operator
$T(q)$ with $q\in L_{2}[0,1]$ to be a spectral operator, in sense of Dunford
[1], one analytic and one geometric. The analytic version was stated in terms
of the solutions of the Hill equation. The geometric version of the criterion
uses algebraic and geometric\ properties of the spectra of the
periodic/antiperiodic and Dirichlet boundary value problems. The problem of
explicitly describing for which potentials $q$ the Schrodinger operators
$T(q)$ are spectral operators has remained open for about 65 years. In [7]
(see also [12, Sect. 2.7], I found explicit conditions on the potential $q$
such that $T(q)$ is an asymptotically spectral operator, using the asymptotic
formulas for the Bloch eigenvalues and Bloch functions. However, since these
asymptotic formulas do not provide any information about the existence of
associated functions corresponding to small eigenvalues, the method in [7]
does not yield any conditions for the spectrality of $T(q)$. Moreover, the
following well-known examples demonstrate that the spectrality of $T(q)$ is a
very rare phenomenon. Therefore, it is natural to conclude that finding
explicit conditions on $q$ that guarantee the spectrality of $T(q)$ is a
complex and generally ineffective problem.

\begin{example}
Consider the case $q(x)=ae^{i2\pi x}$ and $a\neq0.$ The numbers $\left(  \pi
n\right)  ^{2}$ for $n\neq0$ are the spectral singularities (see for example
[12, Sect. 3.3]) for all $n\in\mathbb{Z}$. It means that the operator $T(q)$,
in this case, is not a spectral operator.
\end{example}

\begin{example}
Consider the case $q(x)=ae^{i2\pi x}+be^{-i2\pi x}$ of Mathier-Schrodinger
operator. In [11, Theorem 1] (see also [12, Sect. 4.1]), I proved that the
operator $T(q)$ is an asymptotically spectral operator if and only if $\mid
a\mid=\mid b\mid$\ and
\[
\text{ }\inf_{q,p\in\mathbb{N}}\{\mid q\alpha-(2p-1)\mid\}\neq0,
\]
where $\alpha=\pi^{-1}\arg(ab)$, $\mathbb{N=}\left\{  1,2,...,\right\}  .$
Moreover, Theorem 1 of [11]implies that if $ab\in\mathbb{R}$, then $T(q)$ is a
spectral operator if and only it is self adjoint (see Corollary 1 of [11]).
Thus, in this case, there are no spectral operators that are not self-adjoint.
\end{example}

\begin{example}
Let $q$ be PT-symmetric periodic optical potential $4\cos^{2}x+4iV\sin2x.$
Then in [14] (see also [12, Sect. 5.4]) we prove that the operator $T(q)$ is a
spectral operator if and only if $V=0$, that is, $q(x)=2\cos2x$ and $T(q)$ is
a self-adjoint operator. Thus, in this important case as well, there are no
spectral operators that are not self-adjoint.
\end{example}

Thus, if $n$ is an even number and condition on $p_{1}$ used in Case 2 does
not holds, then one cannot find effective conditions on the coefficients of
(1) that guarantee the spectrality of $L,$ since this case is similar to its
subcase $T(q).$ However, the case where the operator $L$ is generated by the
differential expression (1) with even $n$, $p_{1}$ a real nonzero constant
$c,$ is similar to the odd case in the following sense. In both odd and even
cases, we consider the operator $L_{t}$ as a perturbations of the operators
$L_{t}(0)$ and $L_{t}(c),$ respectively, by the operator generated by the
expression
\begin{equation}
p_{2}(x)y^{(n-2)}(x)+p_{3}\left(  x\right)  y^{(n-3)}(x)+...+p_{n}(x)y \tag{6}%
\end{equation}
and boundary conditions (2), where $L_{t}(0)$ and $L_{t}(c)$ are associated by
the expression $(-i)^{n}y^{(n)}$ and $(-i)^{n}y^{(n)}+cy^{(n-1)},$
respectively. We say that these operators are the main part of $L_{t}$ in the
odd and even cases, respectively. The eigenvalues $\mu_{k}(t,c)$ of the main
part $L_{t}(c)$ are
\begin{equation}
\mu_{k}(t,c)=(2\pi k+\pi t)^{n}+c(2\pi ki+\pi ti)^{n-1} \tag{7}%
\end{equation}
for $k\in\mathbb{Z}.$ These eigenvalues are simple if $c\neq0,$ as the
eigenvalues $(2\pi k+\pi t)^{n}$ of $L_{t}(0)$ if $n$ is an odd number. But
despite this similarity, the examination of the eigenvalues {}{}of the $L_{t}$
operator in the odd and even cases is completely different.

\section{The case of odd order}

In this section, we consider the operators $L$ generated by (1), where $n$ is
an odd number greater that $1$ and $p_{1}$ is the zero function. We use the
following results of [13] and [9], formulated here as Summary 1 and Summary 2, respectively.

\begin{summary}
If $n$ is an odd integer greater than $1$ and (3) holds, then the eigenvalues
of $L_{t}$ lie on the disks
\[
U(k,t)=\left\{  \lambda\in\mathbb{C}:\left\vert \lambda-(2\pi k+\pi
t)^{n}\right\vert <\delta_{k}(t)\right\}
\]
for $k\in\mathbb{Z},$ where
\[
\delta_{k}(t):=\frac{3}{2}\pi^{n-2}C\left\vert (2k+t)\right\vert ^{n-2}.
\]
Moreover, each of these disks contains only one eigenvalue (counting
multiplicities) of $L_{t},$ and the closures of these disks are pairwise
disjoint closed disks
\end{summary}

Note that in [13], we considered differential operators generated by (1) when
coefficient of $y^{(n-v)}$ for $v=2,3,...,n$ was $(-i)^{n-v}p_{v},$ with
$p_{v}$ being a PT-symmetric function. However, the proof of the results in
Summary 1 for the case of this paper remains unchanged. Using this summary, we
obtain the following result.

\begin{theorem}
If $n$ is an odd integer greater than $1$ and (3) holds, then all eigenvalues
of $L_{t}$ for all $t\in(-1,1]$ are simple and there exists a function
$\lambda,$ analytic on $\mathbb{R},$ such that $\sigma(L)=\left\{
\lambda(t):t\in\mathbb{R}\right\}  .$
\end{theorem}

\begin{proof}
It follows from Summary 1 that, all eigenvalues of $L_{t}$ \ for all
$t\in(-1,1]$ are simple. Let us denote the eigenvalue of $L_{t}$ lying in
$U(k,t)$ by $\lambda_{k}(t).$ This eigenvalue is a simple root of the
characteristic equation $\Delta(\lambda,t)=0,$ where
\[
\Delta(\lambda,t)=\det(y_{j}^{(\nu-1)}(1,\lambda)-e^{it}y_{j}^{(\nu
-1)}(0,\lambda))_{j,\nu=1}^{n}=
\]%
\[
e^{in\pi t}+f_{1}(\lambda)e^{i(n-1)\pi t}+f_{2}(\lambda)e^{i(n-2)\pi
t}+...+f_{n-1}(\lambda)e^{i\pi t}+1,
\]
$y_{1}(x,\lambda),y_{2}(x,\lambda),\ldots,y_{n}(x,\lambda)$ are the solutions
of the equation
\[
(-i)^{n}y^{(n)}(x)+p_{2}\left(  x\right)  y^{(n-2)}(x)+p_{3}\left(  x\right)
y^{(n-3)}(x)+...+p_{n}(x)y=\lambda y(x)
\]
satisfying $y_{k}^{(j)}(0,\lambda)=0$ for $j\neq k-1$ and $y_{k}%
^{(k-1)}(0,\lambda)=1,$ and $f_{1}(\lambda),f_{2}(\lambda),...$ are the entire
functions (see [5, Chap. 1]). Let us prove that $\lambda_{k}(t)$ analytically
depend on $t$ in $(-1,1).$ Take any point $t_{0}$ from $(-1,1).$ By Summary 1,
$\lambda_{k}(t_{0})$ is a simple eigenvalue and hence a simple root of the
equation $\Delta(\lambda,t_{0})=0.$ By implicit function theorem, there exist
$\varepsilon>0$ and an analytic function $\lambda(t)$ on $(t_{0}%
-\varepsilon,t_{0}+\varepsilon)$ such that $\Delta(\lambda(t),t)=0$ for all
$t\in(t_{0}-\varepsilon,t_{0}+\varepsilon)$ and $\lambda(t_{0})=\lambda
_{k}(t_{0}).$ It mean that $\lambda(t)$ for $t\in(t_{0}-\varepsilon
,t_{0}+\varepsilon)$ is an eigenvalue of $L_{t}.$ Since the disk $U(k,t)$
continuously depends on $t$ and has no intersection point with the disks
$U(m,t)$ for $m\neq n,$ the number $\varepsilon$ can be chosen so that
$\lambda(t)\in U(k,t)$ for $t\in(t_{0}-\varepsilon,t_{0}+\varepsilon)$ and
hence $\lambda(t)=\lambda_{k}(t).$

Now let us consider the eigenvalue $\lambda_{k}(1).$ Arguing as above and
using the equalities $\Delta(\lambda,t+2)=\Delta(\lambda,t)$ and
$L_{t+2}=L_{t},$ we conclude that there exist $\varepsilon>0$ and an analytic
function $\lambda(t)$ on $(1-\varepsilon,1+\varepsilon)$ such that
$\Delta(\lambda(t),t)=0$ for $t\in(1-\varepsilon,1+\varepsilon)$ and the
following equalities hold: $\lambda(t)=\lambda_{k}(t)$ for $t\in
(1-\varepsilon,1]$ and $\lambda(t)=\lambda_{k+1}(t-2)$ for $t\in
(1,1+\varepsilon).$ Thus, $\lambda_{k+1}(t)$ is the analytic continuation of
$\lambda_{k}(t)$ for all $k\in\mathbb{Z}.$ Therefore, the function
$\lambda(t)$ defined by
\begin{equation}
\lambda(t)=\lambda_{k}(t-2k) \tag{8}%
\end{equation}
for $t\in(2k-1,2k+1]$ depends analytically on $t$ and maps $\mathbb{R}$ onto
$\sigma(L).$
\end{proof}

Now, using the following summary of [9], we consider the projections of
$L_{t}$ and spectrality of $L.$

\begin{summary}
In Case 1 and Case 2 (see the introduction and [9]), there exist positive
constants $N$ and $c(N)$ such that the eigenvalues $\lambda_{k}(t)$ of $L_{t}$
for $\left\vert k\right\vert >N$ are simple and
\begin{equation}
\parallel%
{\textstyle\sum\limits_{k\in J}}
\frac{1}{\alpha_{k}(t)}(f,\Psi_{k,t}^{\ast})\Psi_{k,t}\parallel^{2}\leq
c(N)\left\Vert f\right\Vert ^{2} \tag{9}%
\end{equation}
for all $f\in L_{2}(0,1),$ $t\in(-1,1]$ and $J\subset\left\{  k\in
\mathbb{Z}:\left\vert k\right\vert >N\right\}  ,$ where $\alpha_{k}%
(t)=(\Psi_{k,t},\Psi_{k,t}^{\ast}),$ $\Psi_{k,t}$ and $\Psi_{k,t}^{\ast}$ are
the normalized eigenfunctions of $L_{t}$ and $L_{t}^{\ast}$ corresponding to
the eigenvalues $\lambda_{k}(t)$ and $\overline{\lambda_{k}(t)},$ respectively.
\end{summary}

Let $\gamma$ be a closed contour lying in the resolvent set $\rho(L_{t})$ of
$L_{t}$ and enclosing only the eigenvalues $\lambda_{k_{1}}(t),$
$\lambda_{k_{2}}(t),...,\lambda_{k_{s}}(t)$. It is well-known that (see [5,
Chap. 1]) if these eigenvalues are simple and $e(t,\gamma)$ is the projection
defined by%
\[
e(t,\gamma)=\int_{\gamma}\left(  L_{t}-\lambda I\right)  ^{-1}d\lambda,\text{
}%
\]
then
\[
e(t,\gamma)f=%
{\textstyle\sum\limits_{j=1,2,...,s}}
\frac{1}{\alpha_{k_{j}}(t)}(f,\Psi_{k_{j},t}^{\ast})\Psi_{k_{j},t}.
\]
It is clear that
\begin{equation}
\left\Vert e(t,\gamma)\right\Vert \leq%
{\textstyle\sum\limits_{j=1,2,...,s}}
\frac{1}{\left\vert \alpha_{k_{j}}(t)\right\vert }. \tag{10}%
\end{equation}
In particular, if $\gamma$ encloses only $\lambda_{k}(t),$ where $\lambda
_{k}(t)$ is a simple eigenvalue, then
\begin{equation}
e(t,\gamma)=\frac{1}{\alpha_{k}(t)}(f,\Psi_{k,t}^{\ast})\Psi_{k,t}\text{ }%
\And\left\Vert e(t,\gamma)\right\Vert =\frac{1}{\left\vert \alpha
_{k}(t)\right\vert }. \tag{11}%
\end{equation}
Moreover, $\left\vert \alpha_{k}(t)\right\vert $ continuously depend on $t$
and $\alpha_{k}(t)\neq0$ (see Theorem 2.1 in [10]). Therefore, there exists a
positive constant $c_{k}$ such that%
\begin{equation}
\frac{1}{\left\vert \alpha_{k}(t)\right\vert }<c_{k} \tag{12}%
\end{equation}
for all $t\in(-1,1].$

Now using (9)-(12), we prove the following theorem about spectrality of $L.$

\begin{theorem}
If $n$ is an odd number greater than $1$ and (3) holds, then $L$ is a spectral operator.
\end{theorem}

\begin{proof}
Let $\gamma(t)$ be a closed contour such that $\gamma(t)\subset\rho\left(
L_{t}\right)  $. It follows from Summary 1 and the definition of $\lambda
_{j}(t)$ that $\left\vert \lambda_{j}(t)\right\vert \rightarrow\infty$
uniformly on $(-1,1]$ as $\left\vert j\right\vert \rightarrow\infty.$
Therefore, there exist indices $k_{1},k_{2},\cdot\cdot\cdot,k_{s}$ from
$\left\{  k\in\mathbb{Z}:\left\vert k\right\vert \leq N\right\}  $ and set
$J\subset\left\{  k\in\mathbb{Z}:\left\vert k\right\vert >N\right\}  $ such
that the eigenvalues of $L_{t}$ lying inside $\gamma$ are $\lambda_{j}(t)$ for
$j\in\left(  \left\{  k_{1},k_{2},\cdot\cdot\cdot,k_{s}\right\}  \cup
J\right)  $, where $N$ is defined in Summary 2 and does not depend on $t$.
Then, we have
\[
e(t,\gamma(t))f=%
{\textstyle\sum\limits_{j=1,2,...,s}}
\frac{1}{\alpha_{k_{j}}(t)}(f,\Psi_{k_{j},t}^{\ast})\Psi_{k_{j},t}+%
{\textstyle\sum\limits_{k\in J}}
\frac{1}{\alpha_{k}(t)}(f,\Psi_{k,t}^{\ast})\Psi_{k,t}.
\]
Therefore, it follows from (9), (10), and (12) that, there exists a constant
$M$ such that
\begin{equation}
\left\Vert e(t,\gamma(t))\right\Vert <M \tag{13}%
\end{equation}
for all $t\in(-1,1]$ and $\gamma(t)\subset\rho\left(  L_{t}\right)  .$

Moreover, the system of root functions of $L_{t}$ forms a Riesz basis in
$L_{2}(0,1)$ for all $t\in(-1,1],$ and it follow from Summary 1 that, the
system of root functions is the system of eigenfunctions $\left\{  \Psi
_{k,t}(x):k\in\mathbb{Z}\text{ }\right\}  $; that is, the equality
\[
f=%
{\textstyle\sum\limits_{k\in\mathbb{Z}}}
\frac{1}{\alpha_{k}(t)}(f,\Psi_{k,t}^{\ast})\Psi_{k,t}%
\]
holds for all $f\in L_{2}[0,1]$ and $t\in(-1,1].$ This equality and (13) imply
that the proof of this theorem follows from Theorem 3.5 of [3].
\end{proof}

Now, using spectral expansion obtained in [9], we derive an elegant spectral
expansion for the operator $L$ assuming that $n$ is an odd integer greater
than $1$ and that (3) holds. Since all eigenvalues are simple, the operator
$L$ has no essential spectral singularities (ESS) and the equation (2.18) of
[9] takes the form
\begin{equation}
f(x)=\frac{1}{2}%
{\displaystyle\sum\limits_{k\in\mathbb{Z}}}
\int\limits_{(-1,1]}a_{k}(t)\Psi_{k,t}(x)dt \tag{14}%
\end{equation}
for $f\in L_{2}(-\infty,\infty),$ where $a_{k}(t)=%
{\textstyle\int\limits_{-\infty}^{\infty}}
\frac{1}{\alpha_{k}(t)}f(x)\overline{\Psi_{k,t}^{\ast}(x)}dx.$

\section{The case of even order}

In this section, we consider the operators $L$ generated by (1), where $n$ is
an even number and $p_{1}$ is a nonzero real constant $c$. One can see from
the proof of Theorem 2 that, to prove spectrality, we used Summary 2 and the
simplicity of all eigenvalues of $L_{t}$ for all $t\in(-1,1].$ Summary 2 holds
in the case of even order if $L_{t}$ is generated by (1) and condition on
$p_{1}$ used in Case 2 is satisfied. Since this condition holds when
$p_{1}(x)=c,$ where $c$ is a nonzero real number, we can apply Summary 2.
Therefore, it remains to prove that if $n$ is an even number greater than $1$
and condition (4) is satisfied, then all eigenvalues of $L_{t}$ are simple for
all $t\in(-1,1].$

In Section 2, to establish the simplicity of all eigenvalues, we used Summary
1, which holds only in the odd-order case. Moreover, the proof of Summary 1
does not carry over to the even-order case. That is why, we need to consider
the simplicity of all eigenvalues of $L_{t}.$ To this end, we investigate the
operator $L_{t}$\ as a perturbation of the operator $L_{t}(c)$ (defined in the
introduction), by the operator associated with expression (6). The eigenvalues
$\mu_{k}(t,c)$ of $L_{t}(c)$ are simple and defined by (7). Our goal to prove
that if (4) holds, then the eigenvalues of $L_{t}$ are also simple. For this,
we consider the family of operators
\[
L_{t,\varepsilon}:=L_{t}(c)+\varepsilon(L_{t}-L_{t}(c))
\]
and show that there exists a closed curve $\gamma_{k}$ enclosing only the
eigenvalue $\mu_{k}(t,c)$ which belongs to the resolvent set of
$L_{t,\varepsilon}$ for all $\varepsilon\in\lbrack0,1].$ Since $\gamma_{k}$
encloses only one eigenvalue (counting multiplicity) of $L_{t,0}=L_{t}(c),$ a
standard argument implies that there is exactly one eigenvalue (counting
multiplicity) of $L_{t}=L_{t,1}$ inside $\gamma_{k}$ and this is a simple eigenvalue.

Let $\lambda(k,t,\varepsilon)$ be an eigenvalue of the operator
$L_{t,\varepsilon}$ and let $\Psi_{\lambda(k,t,\varepsilon)}$ be a
corresponding normalized eigenfunction. For brevity, we sometimes write
$\Psi_{\lambda}$ and $\lambda$ instead of $\Psi_{\lambda(k,t,\varepsilon)}$
and $\lambda(k,t,\varepsilon),$ respectively. The normalized eigenfunction of
$L_{t}(c)$ corresponding to the eigenvalue $\mu_{k}(t,c)$ (see (7)) is
$e^{i\left(  2\pi k+\pi t\right)  x},$ where $k\in\mathbb{Z}$ and
$t\in(-1,1].$

From the equation $L_{t,\varepsilon}\Psi_{\lambda}=\lambda\Psi_{\lambda},$
using the obvious equality
\[
(L_{t}(c)\Psi_{\lambda},e^{i\left(  2\pi k+\pi t\right)  x})=((2\pi k+\pi
t)^{n}+c(2\pi ki+\pi ti)^{n-1})(\Psi_{\lambda},e^{i\left(  2\pi k+\pi
t\right)  x}),
\]
we obtain
\begin{equation}
\left(  \lambda-\left(  2\pi k+\pi t\right)  ^{n}-c(2\pi ki+\pi ti)^{n-1}%
\right)  \left(  \Psi_{\lambda},e^{i\left(  2\pi k+\pi t\right)  x}\right)
=\varepsilon\sum\limits_{\nu=2}^{n}(p_{v}\Psi_{\lambda}^{(n-v)},e^{i\left(
2\pi k+\pi t\right)  x}), \tag{15}%
\end{equation}
where $(\cdot,\cdot)$ denotes the inner product in $L_{2}\left[  0,1\right]
.$ Applying integration by parts, we get
\begin{equation}
\left\vert \varepsilon\sum\limits_{\nu=2}^{n}(p_{v}\Psi_{\lambda}%
^{(n-v)},e^{i\left(  2\pi k+\pi t\right)  x})\right\vert \leq CP(k,t),
\tag{16}%
\end{equation}
where
\[
P(k,t)=\left\{
\begin{array}
[c]{c}%
\left(  2\pi k+\pi t\right)  ^{n-2}\text{ if }k\neq0\text{ }\\
\pi^{n-2}\text{ if }k=0\text{ }%
\end{array}
\right.  .
\]
This inequality, together with (15), implies that
\begin{equation}
\left\vert \left(  \Psi_{\lambda},e^{i\left(  2\pi k+\pi t\right)  x}\right)
\right\vert \leq\frac{CP(k,t)}{\left\vert \lambda-\left(  2\pi k+\pi t\right)
^{n}-c(2\pi ki+\pi ti)^{n-1}\right\vert }. \tag{17}%
\end{equation}

Now, using (17), we prove the following lengthy technical lemma for the case
$n>2.$ The case $n=2$ is very simple and will be discussed at the end of the paper.

\begin{lemma}
If $n$ is an even number greater that $2$ and condition (4) is satisfied, then
the horizontal lines
\[
H(n,t,s)=\left\{  (x,y)\in\mathbb{R}^{2}:y=ci^{n-2}((2s+1)\pi+\pi
t)^{n-1}\right\}
\]
for $s=0,\pm1,\pm2,...,$ belong to the resolvent set of $L_{t,\varepsilon}$
for all $\varepsilon\in\lbrack0,1]$ and $t\in\lbrack0,1].$
\end{lemma}

\begin{proof}
\bigskip First, consider the case $s\geq0$. Suppose there exists $\lambda\in
H(n,t,s)$ which is an eigenvalue of $L_{t,\varepsilon}$ for some
$\varepsilon\in\lbrack0,1].$ Then, for the denominator in the expression from
(17), we have the estimate:
\begin{equation}
\left\vert \left(  \lambda-\left(  2\pi k+\pi t\right)  ^{n}-c(2\pi ki+\pi
ti)^{n-1}\right)  \right\vert \geq\left\vert c((2s+1)\pi+\pi t)^{n-1}-c(2\pi
k+\pi t)^{n-1}\right\vert .\tag{18}%
\end{equation}
Using this estimation we prove that
\begin{equation}
\sum_{k\in\mathbb{Z}}\left\vert \left(  \Psi_{\lambda},e^{i\left(  2\pi k+\pi
t\right)  x}\right)  \right\vert ^{2}<1.\tag{19}%
\end{equation}
This contradicts Parseval's equality for the orthonormal basis $\left\{
e^{i\left(  2\pi k+\pi t\right)  x}:k\in\mathbb{Z}\right\}  .$ This means that
$\lambda$ is not an eigenvalue and therefore belong to the resolvent set of
the operators $L_{t,\varepsilon}$. To prove (19), we write the left-hand side
of (19) as the sum of the following four terms: $S_{1}(s),$ $S_{2}(s),$
$S_{2}(s),$ $S_{4}(s),$ and estimate them separately, where
\[
S_{1}(s)=\sum_{k>s}\left\vert \left(  \Psi_{\lambda},e^{i\left(  2\pi k+\pi
t\right)  x}\right)  \right\vert ^{2},\text{ }S_{2}(s)=\sum_{-s\leq k\leq
s,k\neq0}\left\vert \left(  \Psi_{\lambda},e^{i\left(  2\pi k+\pi t\right)
x}\right)  \right\vert ^{2},
\]%
\[
S_{3}(s)=\sum_{k<-s}\left\vert \left(  \Psi_{\lambda},e^{i\left(  2\pi k+\pi
t\right)  x}\right)  \right\vert ^{2},\text{ }S_{4}(s)\left\vert =\left(
\Psi_{\lambda},e^{i\pi tx}\right)  \right\vert ^{2}.
\]
To estimate $S_{1}(s)$ and $S_{2}(s),$ we use the following obvious
inequalities
\begin{equation}
\left\vert a^{n-1}-b^{n-1}\right\vert \geq\left\vert a-b\right\vert \left\vert
a\right\vert ^{n-2},\tag{20}%
\end{equation}
for $ab\geq0$ and for $\left\vert a\right\vert \leq b,$ respectively. If
$k>s\geq0,$ then both $a=:(2\pi k+\pi t)$ and $b:=((2s+1)\pi+\pi t)$ are
positive numbers, if $s\geq0$ and $-s\leq k\leq s,$ then $\left\vert
a\right\vert \leq b.$ Therefore, using (18) and (20), we obtain
\[
\left\vert \left(  \lambda-\left(  2\pi k+\pi t\right)  ^{n}-c(2\pi ki+\pi
ti)^{n-1}\right)  \right\vert \geq\left\vert c\pi(2k-2s-1))(2\pi k+\pi
t)^{n-2}\right\vert .
\]
Substituting this into (17), where $P(k,t)=\left(  2\pi k+\pi t\right)
^{n-2}$ for $k\neq0$ and using the well-known identity
\begin{equation}
\sum_{n=1}^{\infty}\frac{1}{(2n-1)^{2}}=\frac{1}{8}\pi^{2},\tag{21}%
\end{equation}
we obtain the following estimates
\begin{equation}
S_{1}(s)\leq\sum_{k>s}\frac{C^{2}}{c^{2}\pi^{2}\left\vert (2k-2s-1)\right\vert
^{2}}=\frac{C^{2}}{8c^{2}},\text{ }S_{2}(s)<\frac{C^{2}}{8c^{2}}.\tag{22}%
\end{equation}

If $k<0,$ then both $((2s+1)\pi+\pi t)^{n-1}$ and $-(2\pi k+\pi t)^{n-1}$ are
positive numbers for $s\geq0.$ Therefore, from (18), we obtain
\begin{equation}
\left\vert \left(  \lambda-\left(  2\pi k+\pi t\right)  ^{n}-c(2\pi ki+\pi
ti)^{n-1}\right)  \right\vert \geq\left\vert c\right\vert \left\vert (2\pi
k+\pi t)^{n-1}\right\vert . \tag{23}%
\end{equation}
Moreover, $\left\vert 2\pi k+\pi t\right\vert \geq\pi\left(  \left\vert
2k\right\vert -1\right)  $ for $t\in\lbrack0,1].$ Using this inequality (17),
(23) and (21), we obtain
\begin{equation}
S_{3}(s)\leq\sum_{k<-s}\frac{C^{2}}{c^{2}\pi^{2}(\left\vert 2k\right\vert
-1)^{2}}=\frac{C^{2}}{c^{2}\pi^{2}}\left(  \frac{\pi^{2}}{8}-\sum_{-s\leq
k<0}\frac{1}{(\left\vert 2k\right\vert -1)^{2}}\right)  . \tag{24}%
\end{equation}
It remains to estimate $S_{4}(s).$ Using (17), where $P(0,t)=\pi^{n-2},$ and
(18), we obtain
\begin{equation}
S_{4}(s)\leq\frac{\left(  \pi^{n-2}C\right)  ^{2}}{c^{2}\left\vert
((2s+1)\pi+\pi t)^{n-1}-(\pi t)^{n-1}\right\vert ^{2}}. \tag{25}%
\end{equation}
Now, using (22), (24) and (25), we prove that%
\begin{equation}
\sum_{k\in\mathbb{Z}}\left\vert \left(  \Psi_{\lambda},e^{i\left(  2\pi k+\pi
t\right)  x}\right)  \right\vert ^{2}=S_{1}(s)+S_{2}(s)+S_{3}(s)+S_{4}%
(s)<(\frac{1}{4}+\frac{1}{\pi^{2}})\frac{C^{2}}{c^{2}}. \tag{26}%
\end{equation}
for all $s\geq0.$ First we prove (26) for $s=0.$ From (24) and (25) , we
obtain that
\begin{equation}
S_{3}(0)\leq\frac{C^{2}}{8c^{2}},\text{ }S_{4}(0)\leq\frac{\left(  \pi
^{n-2}C\right)  ^{2}}{c^{2}\left\vert (\pi+\pi t)^{n-1}-(\pi t)^{n-1}%
\right\vert ^{2}}\leq\frac{C^{2}}{\pi^{2}c^{2}}, \tag{27}%
\end{equation}
since $(\pi+\pi t)^{n-1}-(\pi t)^{n-1}$ is a nondecreasing function on
$[0,1].$ Moreover, it follows from the definition of $S_{2}(s)$ that
$S_{2}(0)=0.$ Therefore, the inequality in (26), for $s=0,$ follows from (22)
and (27). Now, we prove (26) for $s\geq1.$ It follows from (24) and (25) that
\begin{equation}
S_{3}(s)\leq\frac{C^{2}}{c^{2}\pi^{2}}\left(  \frac{\pi^{2}}{8}-1\right)
,\text{ }S_{4}(s)\leq\frac{\left(  \pi^{n-2}C\right)  ^{2}}{c^{2}\left\vert
(3\pi+\pi t)^{n-1}-(\pi t)^{n-1}\right\vert ^{2}}\leq\frac{1}{9}\frac{C^{2}%
}{\pi^{2}c^{2}} \tag{28}%
\end{equation}
for all $s\geq1$ and $n\geq2,$ since $(3\pi+\pi t)^{n-1}-(\pi t)^{n-1}$ is an
increasing function on $[0,1].$ Instead (27) using (28) and noting that
\[
\frac{1}{\pi^{2}}+\frac{8}{9}\frac{1}{\pi^{2}}>\frac{1}{8},
\]
we see that (26) holds for all $s\geq1.$

Now let us consider the case $s<0.$ If $k<0,$ then both $(2\pi k+\pi t)$ and
$(2s+1)\pi+\pi t$ are nonpositive numbers and we can use (20). Therefore,
arguing as in the proof of (22), we obtain
\begin{equation}
S_{5}(s):=\sum_{k\leq s}\left\vert \left(  \Psi_{\lambda},e^{i\left(  2\pi
k+\pi t\right)  x}\right)  \right\vert ^{2}\leq\frac{C^{2}}{8c^{2}}\text{ ,
}S_{6}(s):=\sum_{s<k<0}\left\vert \left(  \Psi_{\lambda},e^{i\left(  2\pi
k+\pi t\right)  x}\right)  \right\vert ^{2}<\frac{C^{2}}{8c^{2}}.\tag{29}%
\end{equation}
If $k>0,$ then both $((2s+1)\pi+\pi t)^{n-1}$ and $-(2\pi k+\pi t)^{n-1}$ are
nonpositive numbers for $s<0.$ Therefore, from (18) and (17), by using the
inequality $\left\vert 2\pi k+\pi t\right\vert \geq\left\vert 2\pi
k\right\vert $ for $t\in\lbrack0,1]$ and the well-known equality
\[
\sum_{n=1}^{\infty}\frac{1}{(2n)^{2}}=\frac{1}{24}\pi^{2},
\]
we obtain%
\[
\left\vert \left(  \lambda-\left(  2\pi k+\pi t\right)  ^{n}-c(2\pi ki+\pi
ti)^{n-1}\right)  \right\vert \geq\left\vert c(2\pi k+\pi t)^{n-1}\right\vert
\]
and
\begin{equation}
S_{7}(s):=\sum_{k>0}\left\vert \left(  \Psi_{\lambda},e^{i\left(  2\pi k+\pi
t\right)  x}\right)  \right\vert ^{2}\leq\sum_{k>0}\frac{C^{2}}{c^{2}\pi
^{2}(2k)^{2}}=\frac{C^{2}}{24c^{2}}.\tag{30}%
\end{equation}

It remains to estimate $\left\vert \left(  \Psi_{\lambda},e^{i\pi tx}\right)
\right\vert $ for $s<0.$ First consider the case $s=-1.$ Since the expression
$\left\vert (-\pi+\pi t)^{n-1}-(\pi t)^{n-1}\right\vert $ gets its minimum
value at $t=\frac{1}{2},$ we have
\[
\left\vert (-\pi+\pi t)^{n-1}-(\pi t)^{n-1}\right\vert \geq2^{2-n}\pi^{n-1}%
\]
for $s=-1$ and $t\in\lbrack0,1].$ Now, from (18) and (17), where
$P(0,t)=\pi^{n-2},$ we obtain, \
\begin{equation}
\left\vert \left(  \Psi_{\lambda},e^{i\pi tx}\right)  \right\vert ^{2}%
\leq\frac{2^{2n-4}C^{2}}{\pi^{2}c^{2}}. \tag{31}%
\end{equation}
Therefore using (29)-(31) and noting that $S_{6}(-1)=0,$ we see that
\begin{equation}
\sum_{k\in\mathbb{Z}}\left\vert \left(  \Psi_{\lambda},e^{i\left(  2\pi k+\pi
t\right)  x}\right)  \right\vert ^{2}<\left(  \frac{1}{6}+\frac{2^{2n-4}}%
{\pi^{2}}\right)  \frac{C^{2}}{c^{2}} \tag{32}%
\end{equation}
for $s=-1.$

If $s<-1,$ then
\[
\left\vert ((2s+1)\pi+\pi t)^{n-1}-(\pi t)^{n-1}\right\vert \geq\left\vert
(-3\pi+\pi t)^{n-1}-(\pi t)^{n-1}\right\vert \geq(2^{n-1}+1)\pi^{n-1}%
\]
and \
\begin{equation}
\left\vert \left(  \Psi_{\lambda},e^{i\pi tx}\right)  \right\vert ^{2}%
\leq\left(  \frac{1}{(2^{n-1}+1)^{2}\pi^{2}}\right)  \frac{C^{2}}{c^{2}}.
\tag{33}%
\end{equation}
This with (29) and (30) implies that
\begin{equation}
\sum_{k\in\mathbb{Z}}\left\vert \left(  \Psi_{\lambda},e^{i\left(  2\pi k+\pi
t\right)  x}\right)  \right\vert ^{2}<\left(  \frac{7}{24}+\frac{1}%
{(2^{n-1}+1)^{2}\pi^{2}}\right)  \frac{C^{2}}{c^{2}} \tag{34}%
\end{equation}
for $s<-1.$ Thus, by (26), (32) and (34) we have
\[
\sum_{k\in\mathbb{Z}}\left\vert \left(  \Psi_{\lambda},e^{i\left(  2\pi k+\pi
t\right)  x}\right)  \right\vert ^{2}<A\frac{C^{2}}{c^{2}},
\]
for all $\lambda\in H(n,t,s),$ $s\in\mathbb{Z}$ and $t\in\lbrack0,1],$ where
\[
A=\max\left\{  \frac{1}{4}+\frac{1}{\pi^{2}},\frac{1}{6}+\frac{2^{2n-4}}%
{\pi^{2}},\frac{7}{24}+\frac{1}{(2^{n-1}+1)^{2}\pi^{2}}\right\}  =\frac{1}%
{6}+\frac{2^{2n-4}}{\pi^{2}}%
\]
for $n\geq4.$ It means that if (4) holds then (19) is satisfied. The lemma is proved.
\end{proof}

Now, we consider the vertical lines
\[
V(a)=\left\{  (x,y)\in\mathbb{R}^{2}:x=a\right\}
\]
that belong to the resolvent set of the operators $L_{t,\varepsilon}$ for all
$\varepsilon\in\lbrack0,1].$

\begin{lemma}
If $n$ is an even number, then there exists a positive number $M$ such that:

$(a)$ The vertical lines $V(a)$ for $a<-M$ belong to the resolvent set of the
operators $L_{t,\varepsilon}$ for all $t\in(-1,1]$ and $\varepsilon\in
\lbrack0,1]$.

$(b)$ In the cases $\left\vert t\right\vert \in\lbrack0,1/2]$ and $\left\vert
t\right\vert \in(1/2,1],$ respectively, the lines $V(\left(  2\pi
s+\pi\right)  ^{n})$ and $V(\left(  2\pi s\right)  ^{n})$ for $s>M$ belong to
the resolvent set of the operators $L_{t,\varepsilon}$ for all $\varepsilon
\in\lbrack0,1]$.
\end{lemma}

\begin{proof}
$(a)$ Let $\lambda$ be an eigenvalue of $L_{t,\varepsilon}$ for some
$\varepsilon\in\lbrack0,1].$ If $\lambda\in V(a)$ for $a<0$ then it follows
from (17) that
\begin{equation}
\left\vert \left(  \Psi_{\lambda},e^{i\left(  2\pi k+\pi t\right)  x}\right)
\right\vert \leq\frac{C(2\pi k+\pi t)^{n-2}}{\left\vert a\right\vert
+\left\vert \left(  2\pi k+\pi t\right)  ^{n}\right\vert }<\frac{C}{\left(
2\pi k+\pi t\right)  ^{2}} \tag{35}%
\end{equation}
From (35), we obtain that there exists $n_{1}$ such that \
\begin{equation}
\sum_{\left\vert k\right\vert >n_{1}}\left\vert \left(  \Psi_{\lambda
},e^{i\left(  2\pi k+\pi t\right)  x}\right)  \right\vert ^{2}<\frac{1}{2}.
\tag{36}%
\end{equation}
On the other hand, it follows from the first inequality in (35) that there
exists $M$ such that \
\begin{equation}
\sum_{\left\vert k\right\vert \leq n_{1}}\left\vert \left(  \Psi_{\lambda
},e^{i\left(  2\pi k+\pi t\right)  x}\right)  \right\vert ^{2}<\frac{1}{2}
\tag{37}%
\end{equation}
for $\left\vert a\right\vert >M$. Thus, the inequalities (36) and (37) imply
(19), which completes the proof of part $(a)$.

$(b)$ We now prove part $(b)$ for the case $\left\vert t\right\vert \in
\lbrack0,1/2].$ The case $\left\vert t\right\vert \in(1/2,1]$ is similar. If
$\lambda\in V(\left(  2\pi s+\pi\right)  ^{n})$ for $s>0$ and $\left\vert
t\right\vert \in\lbrack0,1/2],$ then we have the inequality
\begin{equation}
\left\vert \left(  \lambda-\left(  2\pi k+\pi t\right)  ^{n}-c(2\pi ki+\pi
ti)^{n-1}\right)  \right\vert \geq\left\vert (\left(  2\pi s+\pi\right)
^{n})-(2\pi k+\pi t)^{n}\right\vert . \tag{38}%
\end{equation}
Since
\[
\left\vert \left(  2\pi s+\pi\right)  -(2\pi k+\pi t)\right\vert \geq\frac
{\pi}{2}%
\]
for all $k\in\mathbb{Z}$ and $\left\vert t\right\vert \in\lbrack0,1/2],$ the
expression
\[
\sum_{k\in\mathbb{Z}}\frac{\left\vert \left(  2\pi k+\pi t\right)
^{n-2}\right\vert ^{2}}{\left\vert (\left(  2\pi s+\pi\right)  ^{n})-(2\pi
k+\pi t)^{n}\right\vert ^{2}}%
\]
is a sufficiently small number for a sufficiently large value of $s.$
Therefore, from (17) and (38), we obtain that (19) holds. This completes the
proof of the lemma.
\end{proof}

Now, we are ready to prove the main result of this section.

\begin{theorem}
If $n$ is an even number and condition (4) holds, then:

$(a)$ All eigenvalues of the operators $L_{t}$ for all $t\in(-1,1]$ are simple.

$(b)$ $L$ is a spectral operator.
\end{theorem}

\begin{proof}
$(a)$ Let $\lambda(t)$ be arbitrary eigenvalue of $L_{t},$ where $t\in
\lbrack0,1]$ Since the strips
\[
S(k,t):=\left\{  (x,y)\in\mathbb{R}^{2}:b(k,t)\leq y\leq b(k+1,t)\right\}
\]
for $k\in\mathbb{Z},$ cover the plane $\mathbb{R}^{2},$ there exists $k$ such
that $\lambda(t)\in S(k,t),$ where $b(k,t)=c((2k-1)\pi+\pi t)^{n-1}.$
Moreover, there exist constants $a$ and $s$ such that $\lambda(t)$ lies within
the rectangle
\[
R(a,s,k,t)=\left\{  (x,y)\in\mathbb{R}^{2}:a<x<c(s,t),\text{ }b(k,t)\leq y\leq
b(k+1,t)\right\}  ,
\]
where $a<-M$, $s>M,$ and
\[
c(s,t)=\left\{
\begin{array}
[c]{c}%
\left(  2\pi s+\pi\right)  ^{n}\text{ for }\left\vert t\right\vert \in
\lbrack0,1/2]\\
c(s,t)=\left(  2\pi s\right)  ^{n}\text{ for }\left\vert t\right\vert
\in(1/2,1]
\end{array}
\right.
\]
with $M$ as defined in Lemma 2. On the other hand, it follows from Lemmas 1
and 2 that the boundary of the rectangle $R(a,s,k,t)$ belongs to the resolvent
set of the operators $L_{t,\varepsilon}$ for all $\varepsilon\in\lbrack0,1]$.
Since $L_{t,\varepsilon}$ forms a holomorphic family with respect to
$\varepsilon$ and the operator $L_{t,0}$ has exactly one eigenvalue in the
rectangle $R(a,s,k,t)$ for $s>\left\vert k\right\vert $, the operator
$L_{t}=L_{t,1\text{ }}$ must also have exactly one eigenvalue (counting
multiplicity) in this rectangle. This means that the eigenvalue $\lambda(t)$
of $L_{t},$ which lies in this rectangles, is simple. Since $\lambda(t)$ was
chosen as an arbitrary eigenvalue of $L_{t},$ this proves part $(a)$ for
$t\in\lbrack0,1].$ In the same way, we prove part $(a)$ for $t\in(-1,0)$

$(b)$ It follows from the proof of $(a)$ that for each $k\in\mathbb{Z}$ the
strip $S(k,t)$ contains exactly one eigenvalue of $L_{t}$. Denoting the
eigenvalue of $L_{t}$ lying in the strip $S(k,t)$ by $\lambda_{k}(t),$ and
repeating the arguments used in the proof of Theorem 2, we obtain the proof of
part $(b).$
\end{proof}

Note that, by arguing as in the proof (14), we find that (14) remain valid if
$n$ is an even integer greater than $1$ and condition (4) holds.

Now let us consider the case $n=2.$ In this case the operators $L$ and $L_{t}$
are redenoted by $T(c,q)$ and $T_{t}(c,q),$ respectively (see introduction).
In this case (17) and (18) have the following forms
\begin{equation}
\left\vert \left(  \Psi_{\lambda},e^{i\left(  2\pi k+\pi t\right)  x}\right)
\right\vert \leq\frac{\left\Vert q\right\Vert }{\left\vert \lambda-\left(
2\pi k+\pi t\right)  ^{2}-c(2\pi ki+\pi ti)\right\vert }, \tag{39}%
\end{equation}

and
\begin{equation}
\left\vert \left(  \lambda-\left(  2\pi k+\pi t\right)  ^{2}-c(2\pi ki+\pi
ti)\right)  \right\vert \geq\left\vert c\pi\right\vert \left\vert
2s+1-2k\right\vert , \tag{40}%
\end{equation}
respectively, where (40) holds if $\lambda\in H(2,t,s)$ (see Lemma 1 for the
definition of $H(2,t,s)$). Therefore, using (21), we obtain
\[
\sum_{k\in\mathbb{Z}}\left\vert \left(  \Psi_{\lambda},e^{i\left(  2\pi k+\pi
t\right)  x}\right)  \right\vert ^{2}\leq\frac{\left\Vert q\right\Vert ^{2}%
}{\pi^{2}c^{2}}\left(  \sum_{k\leq s}\frac{1}{\left\vert 2s+1-2k\right\vert
^{2}}+\sum_{k>s}\frac{1}{\left\vert 2s+1-2k\right\vert ^{2}}\right)
=\frac{\left\Vert q\right\Vert ^{2}}{4c^{2}}.
\]
This implies that if condition (5) holds, then (19) is satisfied; that is,
$H(2,t,s)$ belong to the resolvent set of the operators $T_{t,\varepsilon}$
for all $\varepsilon\in\lbrack0,1]$ and $t\in\lbrack0,1],$ where
$T_{t,\varepsilon}=T_{t}(c,0)+\varepsilon(T_{t}(c,q)-T_{t}(c,0))$. Therefore,
instead of Lemma 1 using this statement, we obtain the following consequence
of Theorem 3:

\begin{corollary}
If condition (5) holds, then all eigenvalues of the operators $T_{t}(c,q),$
for $t\in\lbrack0,1],$ are simple, and $T(c,q)$ is a spectral operator.
\end{corollary}


\begin{thebibliography}{99}                                                                                               %


\bibitem {}N. Dunford and J. T Schwartz, Linear Operators, Part III: Spectral
Operators. Wiley-- Interscience, New York, 1988.

\bibitem {}F. Gesztesy and V. Tkachenko, A criterion for Hill operators to be
spectral operators of scalar type, J. Analyse Math\textit{.} \textbf{107},
287--353 (2009).

\bibitem {}D. C. McGarvey, Differential operators with periodic coefficients
in $L_{p}(-\infty,\infty)$, J. Math. Anal. Appl., \textbf{11}, 564-596 (1965).

\bibitem {}D. C. McGarvey, Perturbation results for periodic differential
operators, J. Math. Anal. Appl., \textbf{12} 187-234 (1965).

\bibitem {}M. A. Naimark, Linear Differential Operators, Part 1, George G.
Harrap, London, 1967.

\bibitem {}F. S. Rofe-Beketov, The spectrum of nonself-adjoint differential
operators with periodic coefficients. Sov. Math. Dokl. {4}, 1563-1564 (1963).

\bibitem {}O. A. Veliev, Asymptotic analysis of non-self-adjoint Hill
operators, Central European Journal of Mathematics, \textbf{11 (}12),
2234-2256 (2013).

\bibitem {}O. A. Veliev, On the Spectral Singularities and Spectrality of the
Hill Operator, Operators and Matrices, 10 (1), 57-71 (2016).

\bibitem {}O. A. Veliev, Asymptotically Spectral Periodic Differential
Operator, Math. Notes, 104, 364--376 (2018).

\bibitem {}O. A. Veliev, Spectral expansion series with parenthesis for the
nonself-adjoint periodic differential operators, Comm. Pure and Appl. Anal.,
\textbf{18}, 397-424 (2019).

\bibitem {}O. A. Veliev, On the spectrality and spectral expansion of the
non-self-adjoint Mathieu-Hill operator, Comm. Pure and Appl. Anal.,
\textbf{19}, 1537-1562 (2020).

\bibitem {}O. A. Veliev, Non-self-adjoint Schr\"{o}dinger Operator with a
Periodic Potential, Springer, Switzerland, 2021.

\bibitem {}O. A. Veliev, On the spectrum of the differential operators of odd
order with PT-symmetric coefficients, Funktsional'nyi Analiz i ego
Prilozheniya \textbf{59} (2), 17-24 (2025).

\bibitem {}O. A. Veliev, Spectral Analysis of the Schrodinger Operator with a
PT-Symmetric Periodic Optical Potential, J. Math. Phys. 61, 063508 (2020).
\end{thebibliography}
\end{document}